\documentclass[10pt]{article}
\usepackage{amsmath, amsthm, latexsym, amssymb,url}
\usepackage{amsfonts}
\usepackage{inputenc}
\usepackage{epsfig}
\usepackage{xcolor}
\usepackage{graphicx}
\usepackage{yhmath}
\usepackage{mathdots}
\usepackage{mathtools} 
\usepackage{pifont}
\usepackage{mathtools}

\usepackage{tikz}
\input xy
\xyoption{all}

\newcommand{\shrinkmargins}[1]{
  \addtolength{\textheight}{#1\topmargin}
  \addtolength{\textheight}{#1\topmargin}
  \addtolength{\textwidth}{#1\oddsidemargin}
  \addtolength{\textwidth}{#1\evensidemargin}
  \addtolength{\topmargin}{-#1\topmargin}
  \addtolength{\oddsidemargin}{-#1\oddsidemargin}
 \addtolength{\evensidemargin}{-#1\evensidemargin}
  }

\shrinkmargins{.7}

\theoremstyle{plain}

\newtheorem{theorem}{Theorem}[section]
\newtheorem{corollary}[theorem]{Corollary}
\newtheorem{lemma}[theorem]{Lemma}
\newtheorem{proposition}[theorem]{Proposition}

\newtheorem*{teo}{Theorem}

\newtheorem{definition}[theorem]{Definition}

\theoremstyle{remark}
\newtheorem{remark}[theorem]{Remark}

\theoremstyle{definition}
\newtheorem{example}[theorem]{Example}

\newcommand{\Z}{\mathbb{Z}}
\newcommand{\R}{\mathbb{R}}
\newcommand{\Q}{\mathbb{Q}}

\def\L{{\mathcal L}}

\def\v{{\mathbf v}}

\def\bv{{\boldsymbol v}}

\def\bx{{\boldsymbol x}}

\def\bo{{\boldsymbol 0}}

\def\spn{\operatorname{span}}

\def\dim{\operatorname{dim}}
\def\det{\operatorname{det}}

\def\ker{\operatorname{ker}}

\def\SL{\operatorname{SL}}
\def\SO{\operatorname{SO}}

\DeclareMathOperator{\vol}{vol}

\begin{document}

\thispagestyle{empty}
\setcounter{tocdepth}{7}

\title{Bases of minimal vectors in tame lattices.}
\author{Mohamed Taoufiq Damir\thanks{M. T. Damir's work was supported in part by the Academy of Finland, under Grant No. 318937 (Aalto University PROFI funding to C. Hollanti). }, \ Guillermo Mantilla-Soler\thanks{G. Mantilla -Soler's work was supported in part by the Aalto Science Institute.} }


\date{}

\maketitle

\begin{abstract}

Motivated by the behavior of the trace pairing over tame cyclic number fields, we introduce the notion of tame lattices. Given an arbitrary non-trivial lattice $\L$ we construct a parametric family of full-rank sub-lattices $\{\L_{\alpha}\}$ of $\L$ such that whenever $\L$ is tame each $\L_{\alpha}$ has a basis of minimal vectors. Furthermore, for each $\L_{\alpha}$ in the family  a basis of minimal vectors is explicitly constructed.  

\end{abstract}


 
 \section{Introduction}

The purpose of this paper is to introduce the notion of {\it tame lattices}, and to prove the following:

\begin{teo}[cf. Theorem \ref{main}]
Let $\L \subseteq \R^{N}$ be a tame lattice with Lagrangian basis $\{e_{1},...,e_{N}\}$. Let $a:=\langle e_{1}, e_{1}\rangle$ and $h=-\langle e_{1}, e_{2}\rangle.$  Let $r, s$ be integers such that  $0 \neq |r| <N$ and let $m:=r+sN.$ Suppose that 
\[\frac{Na-1}{N^2-1}\leq \left(\frac{m}{r}\right)^2\leq \frac{(aN-1)(N+1)}{N-1}.\]

The lattice $\L_{\v_1}^{(r,s)}$ is a sub-lattice of $\L$ of index $m|r|^{N-1}$, minimum \[\lambda_1(L_{\v_1}^{(r,s)})=ar^2+\frac{m^2-r^2}{N}\] and with basis of minimal vectors \[\{re_1+s\v_1,re_2+s\v_1,\dots,re_N+s\v_1\}.\]
\end{teo}

For a precise definition of tame lattice see Definition \ref{Lagrangian}.  However, a quick way to picture a tame lattice is that it is a lattice with a Gram matrix of the form 

\[\left[
  \begin{array}{cccc}
    a & -h & \ldots & -h \\
   -h & \ a & \ddots & \vdots \\
    \vdots & \ddots & \ddots & -h \\
    -h & \ldots & -h & \ a \\
  \end{array}
\right]\,.\]

for some $a,h$ such that $a-h(N-1)=1$.

\subsection{Artihmetic motivation}

It is our believe that providing the path as to how we got our definitions and  results is an important part of the clarification  and development of mathematical ideas. We briefly show  how the ideas, statements and definitions of this paper came to be.\\

Let $p$ be an odd prime and let $K$ be a Galois number field of degree $p$. In \cite[IV.8]{CoPe} Conner and Perlis showed that if no rational prime has wild ramification in $K$, i.e., $K$ is tame, there exists an integral basis  of $O_
{K}$ such that 

\begin{enumerate}

\item $e_{1}+...+e_{p}=1$.
\item ${\rm Tr}_{K/\Q} (e_{i})  =1$ for all $1 \leq i \leq p$.
\item ${\rm Tr}_{K/\Q} (e_{i}e_{i})  =  {\rm Tr}_{K/\Q} (e_{j}e_{j})$ for all $1 \leq i, j \leq p$.
\item $ {\rm Tr}_{K/\Q} (e_{i}e_{j}) = {\rm Tr}_{K/\Q} (e_{k}e_{l})$ for all $1 \leq i, j, k,l \leq p$ with $i\neq j$ and $k \neq l$.
\end{enumerate}

Conner and Perlis coined the term {\it Lagrangian basis} for a basis of $O_{K}$ satisfying the conditions above, this apparently was inspired by some definitions of Hilbert (see \cite[IV.8.1]{CoPe}).

Recently in n \cite{sueli} and  \cite{oliviera} the authors have used a version of the result above to construct bases of well rounded lattices defined by certain sub-modules of the ring of integers of a prime degree Galois number field. Suppose that $p$ is an odd prime and that $K$ is a degree $p$ tame  Galois number field. In \cite{sueli} and  \cite{oliviera} a set of conditions on positive integers $m \equiv 1 \pmod{p}$  are given so that the sub-lattice $\{x \in O_{K}: {\rm Tr}_{K/\Q}(x) \equiv 0\pmod{m}\}$ of $O_{K}$  has a minimal basis.

Such construction can be generalized to a bigger family of number fields by studying when a notion of Lagrangian basis exists. In fact, the results in \cite{BoMa} show that when $K$ is abelian, not necessarily of prime degree but of prime conductor, a Lagrangian basis exists. Since there are number fields with Lagrangian basis and with neither prime conductor or degree, see for instance example \ref{ex3}, the construction of \cite{sueli} and  \cite{oliviera} can be generalized even further to a bigger class of number fields. 

The concept of tame lattices (see \S\ref{TameLattices}) came by axiomatizing the linear map ${\rm Tr}_{K/\Q}: O_{K} \to \Z$ and conditions (1)-(4). Our definition generalizes the cases of number fields previously mentioned, and at the same time has proven to be useful to construct several interesting lattices beyond the ring of integers of number fields. For instance, the cubic lattice, the checkboard lattice, and $\mathbb{A}_{n}$ all can be obtained as examples of a  construction related to such lattices. In a further work to appear elsewhere we show how their duals, $E_{8}$ and other lattices can be constructed following the ideas developed here.

\begin{remark}
Suppose that $K$ is a tame Galois number field of prime degree.  If one uses the Theorem \ref{main} with  $r=1$, applied to the tame lattice $O_{K}$, the results of \cite{sueli} and  \cite{oliviera} are recovered.
 \end{remark}

\subsection{Outline of the paper} 
 
The paper is organized as follows. Section \ref{BackgroundLattices} introduces the necessary background used throughout the article. In Section \ref{sub}, we first construct a sub-lattice of a given lattice using a co-dimension one subspace inspired by the trace map of a field extension. Using this, we obtain a specific family of full-rank sub-lattices $\L_{\v_1}^{(r,s)}$ of a given lattice $\L$ (see Definition \ref{mainsublat}).  We then present particular examples of the sub-lattices $\L_{\v_1}^{(r,s)}$ (Lemma \ref{Cong1}), where $\L$ is obtained as the Minkowski embedding of the ring of integers of a tame number field. The rest of Section \ref{sub} is dedicated to introducing and calculating some invariants of tame lattices.
In Section \ref{Condi}, we study $\lambda_1(L_{\v_1}^{(r,s)})$, the length of the shortest non-zero vector in $\L_{\v_1}^{(r,s)}$, where $\L$ is a tame lattice. We conclude the section by proving Theorem \ref{main}, where we provide conditions on the sub-lattices $\L_{\v_1}^{(r,s)}$, i.e., $r$ and $s$ for which $\L_{\v_1}^{(r,s)}$ has a minimal basis.

\section{General background on Lattices}\label{BackgroundLattices}

A \textit{lattice} $\L$ in $\mathbb{R}^N$ is a discrete additive subgroup of $\mathbb{R}^N$. If $t$ is the dimension of the sub-space generated by $\L$, it can be shown that  $\L$ is a free $\Z$-module of rank $t$.  In other words, a rank $t$ lattice in $\R^{N}$ is a set of  the form
\begin{equation}
\label{def}
  \L=\Big\{\sum_{i=1}^{t}a_i e_i~|~a_i\in\mathbb{Z}\Big\},   
\end{equation}

where $\{e_1,\dots,e_t\} \subseteq \R^{N}$  is a set of independent vectors. We say that $\L$ is {\it full rank} if $t=N$.
\begin{remark}
In this paper we will deal mostly with full lattices. Thus, from this point whenever we say lattice we mean full unless we explicitly say otherwise. 
\end{remark}
We call  $ M:= [e_1 | \cdots | e_n]$  a \textit{generator matrix} of $\L$, where the vectors $e_1,\dots,e_n$ are viewed as column vectors, so that $\L=M\Z^{N}$.  The matrix $ G=M^T M$ is called a \textit{gram matrix} of $\L$. The lattice $\L$ is said to be {\it integral} if the matrix $G$ has all its entries in $\Z$. If $\L$ is a full lattice, its {\it volume} is defined as $\vol(\L):=\sqrt{\det(G)}$. This definition is independent of the choice of Gram matrix. It can be shown that for  $\L'$ a full sub-lattice of $\L$ we have that  \[[\L:\L']=\frac{\vol(\L')}{\vol(\L)}.\]

An important invariant in the study of the sphere packing problem is the \textit{center density} of $\L$ defined by
\[\delta(\L):=\frac{\lambda_1(\L)^{N/2}}{2^N\vol(\L)}.\]

Given a lattice $\L$, the quantity $\displaystyle \lambda_1(\L):= \min_{\bx \in \L \setminus \{\bo\}} \|\bx\|^2$ is called \textit{the minimum of }$\L$. Finally, the set of {\it minimal vectors} in $\L$ is the set of vectors of minimum norm, i.e., 

\begin{equation}
S(\L) := \{x\in \L: ||x||^2 = \lambda_1(\L) \}.
\end{equation}

The cardinality of $S(\L)$ is known as the \textit{kissing number} of $L$. For a detailed exposition on lattices we refer the reader to \cite{conway2013sphere}.

\subsection{Well-rounded lattices}\label{TopoLattices}

In \cite{conway1995lattice} Conway and Sloane constructed the first example of an $11$-dimensional lattice that is generated by its minimal vectors but in which no set of $11$ minimal vectors form a basis. This construction implies that such lattices exist in every dimension $N\geq 11$. In \cite{martinet1} and \cite{martinet2}, Martinet showed that a lattice of dimension $N\leq 8$, which is generated by its minimal vectors also has a basis of minimal vectors. This study was completed in \cite{martinet3} by Martinet and Schürmann, where the authors showed a similar result for $9$-dimensional lattices and provided a counter-example in dimension $10$.  Thus, the $\Z$-span of $S(\L)$ the set of minimal vectors in a lattice $\L$ can span $\L$ without containing a basis. If $\R \otimes \L=\spn_{\R}(S(\L))$, $\L$ is called \textit{well-rounded.} A stronger condition is when $\L=\spn_{\Z}(S(\L))$, in this case $\L$ is called \textit{strongly well-rounded}. An even stronger condition is when $\L$ has a basis of minimal vectors, that is when  $S(\L)$ contains a basis of $\L$.  Computationally, checking well-roundedness, the weakest of the three conditions mentioned above, for an arbitrary lattice is an NP-hard problem \cite{khot2005hardness}. Deciding whether or not $\L$ is well-rounded  is equivalent to finding $S(\L)$ the set of minimal vectors in $\L$. For a lattice $\L$ of dimension at most $4$, well-roundedness, strong well-roundedness, and having a minimal basis are equivalent conditions (see for instance \cite{pohst1981computation}); however, as mentioned previously, for higher dimensions this is not always the case.\\

\begin{example}
The following lattices are examples of lattices with a minimal basis \cite{conway2013sphere}.
\begin{enumerate}
 \item The orthogonal lattice $\Z^n$.
    \item The checkerboard lattice $$\mathbb{D}_n=\Big\{(x_1,\dots,x_n)\in\Z^n~:~\sum_{i=1}^n x_i\equiv 0\pmod2\Big\}.$$
    \item The $\mathbb{A}_n$ lattice 
    $$\mathbb{A}_n=\Big\{(x_1,\dots,x_{n+1})\in\Z^{n+1}~:~\sum_{i=1}^{n+1} x_i=0\Big\}.$$

\end{enumerate}
\end{example}

Every SWR(strongly well-rounded) lattice is WR(well-rounded). The following example illustrates that the converse does not always hold. To see examples of lattices that are SWR but that have no minimal basis see \cite{conway1995lattice}.
\begin{example} 
Let $N, k$ be positive integers with $N>k^2>1$. Let $\L$ be the lattice generated by $\Z^N$ and the vector $v=(1/k,\dots,1/k)$. By the hypothesis on $N$ and $k$, $S(\L)$ consists of the standard basis vectors, i.e, $\pm e_i=(0,\dots,\pm 1,\dots,0)$ and $\Z^{N} \neq \L $. In particular, $\L$ is WR. On the other hand $\spn_{\Z}(S(L))=\Z^N\neq \L$ so $\L$ is not SWR.

\end{example}

Well-rounded lattices appear in various arithmetic and geometric problems. In particular, in discrete geometry, number theory, and topology. For example, a classical theorem due to Voronoi \cite{voronoi} implies that the local maxima of the sphere packing function are all realized at well-rounded lattices. 
In \cite{mcmullen} well-rounded lattices have been investigated in the context of Minkowski conjecture. Furthermore, topological properties of the set of well-rounded lattices have also been of interest. To name just a few examples, Ash in \cite{ash} proved that the space of all (unimodular) lattices retracts to the space of well-rounded lattices. More recently, a result by Solan \cite{solan2019stable} state that for any lattice $\L$ there exists a diagonal unimodular real matrix $a$ with positive entries such that $a\cdot L$ is a well-rounded lattice.

In addition to their arithmetic and geometric appeal, well-rounded lattices are studied in communication theory. In particular, in physical layer communication reliability \cite{gnilke} and security \cite{Damir}. 

It is also worth mentioning that most of the lattice-based cryptographic protocols lays on the hardness of the shortest vector problem. On the other hand, the problem of determining all the successive minima of an arbitrary lattice is believed to be strictly harder \cite{micciancio}. However, if the lattice is well-rounded, these two problems are equivalent.
 
Hence, from a theoretical point of view, as well as a practical, it is of interest to explicitly construct well-rounded lattices and to study when a given lattice has a well-rounded sub-lattice or a sub-lattice generated by its minimal vectors. It turns out that among all lattices, well-rounded lattices are scarce. So in a probabilistic sense, such lattices are difficult to find.
Studying the geometric structure of well-rounded sub-lattices, strongly well-rounded, and lattices with a minimal basis is a non-trivial question that has been investigated by several authors; for instance, in dimension $2$ work in this direction has been done in \cite{fukshansky}, \cite{baake}, and \cite{kuhnlein}.

We define two lattices $\L$ and $\L'$ to be \textit{similar} if we can obtain $\L$ from $\L'$ using a rotation and a real dilation of $\L'$. It is clear that the well-roundedness property is invariant under similarity, so it is natural to consider the space of lattices of fixed volume $1$, namely,  
\[\mathcal{S}_N= \SO_N(\R)\backslash \SL_N(\mathbb{R})/\SL_N(\mathbb{Z}).\]

It is well-known that $\mathcal{S}_N$ has a unique measure $\mu_N$ (Haar measure) that is right $\SL_N(\mathbb{R})$-invariant. In fact, the set of WR lattices in $\mathcal{S}_N$ has measure zero with respect to $\mu_N$.  Given a lattice $\L$, the $i^th$ successive minimum is
\[\lambda_i(\L) = \inf\{r~|~ \dim({\rm span}(\L\cap \overline{B}(0, r))\geq i\},\]
where $B(0, r)$ is the closed ball of radius $r$ around $0$. \\

It follows from definitions that a lattice $\L$ is WR if and only if all its successive minima are equal.  With this in mind, we can see the set of WR lattices of dimension $N$ is a set defined by $N-1$ (successive minima) equalities. This shows that the space of WR lattices is not a full-dimensional space in $\mathcal{S}_N$, and in particular it has zero measure. The following figure illustrates $\mathcal{S}_2$ and $\mathcal{W}_2$ the sets of similarity classes and well-rounded planar lattices \cite{damir2019well}.
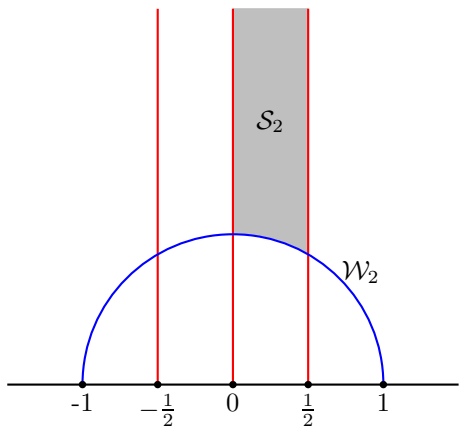
\begin{figure}[h]
\centering
\begin{tikzpicture}\label{w2}
    \filldraw[color=gray!50](3,2)--(3,5)--(4,5)--(4,1.75) arc[x radius=2, y radius=2, start angle=60, end angle=90]--cycle;
     \draw[red, thick](2,0)--(2,5);
     \draw[red, thick](3,0)--(3,5);
     \draw[red, thick](4,0)--(4,5);
     \draw[thick,blue] (5,0) arc[x radius=2, y radius=2, start angle=0, end angle=180];
     \draw[thick] (0,0)--(6,0);
     
     \filldraw[thick](1,0) circle (1pt) node[anchor=north]{-1};
     \filldraw[thick](2,0) circle (1pt) node[anchor=north]{$-\frac{1}{2}$};
     \filldraw[thick](3,0) circle (1pt) node[anchor=north]{0};
     \filldraw[thick](4,0) circle (1pt) node[anchor=north]{$\frac{1}{2}$};
     \filldraw[thick](5,0) circle (1pt) node[anchor=north]{1};
     \node(A) at (3.5,3.5){$\mathcal{S}_2$};
     \node(B) at (4.7,1.5){$\mathcal{W}_2$};
     \end{tikzpicture}
           \caption{Similarity classes of WR planar lattices.}
\end{figure}

\section{A generic construction of a family of sub-lattices of a given lattice}\label{sub}

In this section we construct a family of sub-lattices of a given lattice. Such a construction, when specialized to tame lattices, will lead to a family of lattices that have a minimal basis.

\subsection{Sub-lattices from co-dimension one linear maps}

Let $N$ be a positive integer and let  $\L\subset\R^N$ be a full lattice. Suppose that ${\rm T}: \L \to \Z$ is a non-trivial linear map. Let $\v_{1} \in \L \setminus \ker {\rm T}$.

\begin{proposition}\label{A} Let $r, s$ be some integers and let  $m:=r +s{\rm T}(\v_{1})$. The map 
\begin{align*}
\Phi_{(r,s)}: \L&\rightarrow \L\\
x &\mapsto rx+s{\rm T}(x)\v_1
\end{align*}
is linear. Moreover, ${\rm T} \circ \Phi_{(r,s)} = [m] \circ {\rm T}$ where $[m]: \Z \to \Z$ is the usual multiplication by $m$ map. In particular, if $r$ and $m$ are different from zero, the map $\Phi_{(r,s)}$ is an injection.
\end{proposition}

\begin{proof}
Since {\rm T} is linear, $\Phi_{(r,s)}$ is a composition of linear maps; hence it is linear as well. To verify the second claim let $x \in \L$. Then,  \[{\rm T}(\Phi_{(r,s)}(x))={\rm T}(rx +s{\rm T}(x)\v_{1})=r{\rm T}(x)+s{\rm T}(\v_{1}){\rm T}(x)= m {\rm T}(x).\] For the last claim let $x \in \ker({\rm T})$. Since $m \neq 0$ it follows from the second claim that ${\rm T}(x)=0$ thus, by the definition, $\Phi_{(r,s)}(x)=rx$. Since $\Phi_{(r,s)}(x)=0$, and $r\neq 0$, it follows that $x=0$. 

\end{proof}

\begin{corollary}\label{CoroA}
Let $\L \subseteq \R^N$ be a full lattice, let ${\rm T} \in {\rm Hom}_{\Z}(\L, \Z) \setminus \{0\} $ and let $\v_{1} \in \L \setminus \ker {\rm T}$. Let $r, s$ be integers and define $m:=r +s{\rm T}(\v_{1})$. Suppose that $r$ satisfies 
\begin{itemize}

\item $r \neq 0$

\item $|r| <  | {\rm T}(\v_{1}) |.$
\end{itemize}
Then $\Phi_{(r,s)}(\L)$ is a full sub-lattice of $\L$. Furthermore, \[ \Phi_{(r,s)}(\L) \subseteq \L^{(m)}_{{\rm T}}:= \{x \in \L : {\rm T}(x) \equiv 0 \pmod{mn_{T}}\}\] where $n_{T}$ is the size of the co-kernel of $T$, i.e., $[\Z: {\rm T}(\L)].$

\end{corollary}

\begin{proof}

This is just a reformulation of Proposition \ref{A}. By hypothesis, we have that $r$ and $m$ are non-zero. From the proposition we see that $\Phi_{(r,s)}(\L)$ and $\L$ have the same $\Z$-rank. The claimed inclusion follows from  $T \circ \Phi_{(r,s)} = [m] \circ {\rm T}$.

\end{proof}

\begin{remark}\label{ConditionsOnr} By definition $m$ is defined in terms of $r,s$ and ${\rm T}(\v_{1})$. The conditions of $r$ defined in the corollary are there to guarantee the reverse;   in this case  we have that the lattice $\Phi_{(r,s)}(\L)$  is determined by $m$ and ${\rm T}(\v_{1})$ since $r$ is congruent to $m$ modulo   ${\rm T}(\v_{1})$.
\end{remark}

\begin{lemma}\label{CongruencemLattice}
Let $\L \subseteq \R^N$ be a full lattice, let ${\rm T} \in {\rm Hom}_{\Z}(\L, \Z) \setminus \{0\} $ and let $n_{T}$ be the order of the co-kernel of ${\rm T}$. For a given positive integer $m$ let $\L^{(m)}_{{\rm T}}:= \{x \in \L : {\rm T}(x) \equiv 0 \pmod{mn_{{\rm T}}}\}$. Then,  $\L^{(m)}_{{\rm T}}$ is a full sub-lattice of $\L$ of index $m$.
\end{lemma}

\begin{proof}
We claim that  $\L/\L^{(m)}_{{\rm T}} \cong \Z/m\Z$, from which the result follows. Composing ${\rm T}$ with the reduction map $\Z  \to \Z/mn_{{\rm T}}\Z$ we get a linear map $\L  \to \Z/mn_{{\rm T}}\Z $ with image  $n_{{\rm T}}\Z/mn_{{\rm T}}\Z  \cong  \Z/m\Z$ and with kernel equal to $\L^{(m)}_{{\rm T}}$, hence proving our claim.
\end{proof}

\begin{definition}\label{mainsublat} Let $\L, {\rm T }$, and $\v_{1}$ be as above. Let $r, s$ be integers. The lattice  $\L_{{\rm T},\v_1}^{(r,s)}$ is defined as the sub-lattice of $\L$ given by the image of $\Phi_{(r,s)}$ i.e., 
\[\L_{{\rm T},\v_1}^{(r,s)} := \Phi_{(r,s)}(\L).\]

\end{definition}

It follows from Corollary \ref{CoroA} and Lemma \ref{CongruencemLattice} that for $m=r+s{\rm T}(\v_{1})$ if $r\neq 0, |r| < |{\rm T}(\v_{1})|$,  we have a sequence of rank $N$ lattices \[ \L_{{\rm T},\v_1}^{(r,s)} \subseteq  \L^{(m)}_{{\rm T}} \subseteq \L \] so that $[\L:\L_{ {\rm T},\v_1}^{(r,s)}]$ is a multiple of $m$.  It turns out that such a multiple is $|r|^{N-1}.$

\begin{proposition}\label{Index}
Let $\L, {\rm T }$, and $\v_{1}$ be as above. Let $r, s$ be integers with $r\neq 0$ and $|r| < |{\rm T}(\v_{1})|$ and let $m:= r+s {\rm T}(\v_{1})$.
Then, \[[\L:\L_{ {\rm T},\v_1}^{(r,s)}]=m|r|^{N-1}.\] In particular, $\L_{{\rm T},\v_1}^{(r,s)} =  \L^{(m)}_{{\rm T}}$ if and only if $r=\pm 1$.
\end{proposition}

\begin{proof}
First, suppose that there exists a basis  $\mathcal{B}=\{e_1,\dots,e_N\}$  of $\L$ that is rigid with respect to  ${\rm T}$ i.e., ${\rm T }(e_{i})={\rm T }(e_{j})$ for all  $1\leq  i,j \leq N$. Thanks to Proposition \ref{A} and Corollary \ref{CoroA} the set $\mathcal{B}_{\Phi_{(r,s)}}=\{\Phi_{(r,s)}(e_1),\dots,\Phi_{(r,s)}(e_N)\}$ is a basis for  $L_{{\rm T },\v_{1}}^{(r,s)}$. Hence, the index $[L:L_{{\rm T },\v_1}^{(r,s)}]$ is equal to $|\det(A)|$ where $A$ is the representation matrix of the basis  $\mathcal{B}_{\Phi_{(r,s)}}$ in terms of the basis  $\mathcal{B}$. If $\displaystyle \v_1=\sum_{j=1}^{N}a_j e_j$, then for all  $1 \leq i \leq N$, we have that \[\Phi_{(r,s)}(e_i) = re_i +s{\rm T}(e_i)\sum_{j=1}^{N}a_j e_j,\] hence \[
A = \begin{bmatrix} 
    r+sa_1{\rm T }(e_1) & s a_1{\rm T }(e_2) & \dots &  & & sa_1{\rm T }(e_N) \\
    sa_2{\rm T }( e_1) & r+sa_2{\rm T }( e_2) & \dots & & & sa_2{\rm T }(e_N) \\
    \vdots & &  \ddots &  & & \vdots  \\
        &  & &  & & sa_{N-1}{\rm T }(e_N) \\
      sa_N{\rm T }( e_1) &   s a_N{\rm T }(e_2)  & \dots  & &  & r+s a_N{\rm T }( e_N)
    \end{bmatrix}.
\] Since  $\mathcal{B}$ is rigid with respect to ${\rm T}$, the matrix $A$ is equal to  \[
A = \begin{bmatrix} 
    r+sa_1{\rm T }(e_1) & s a_1{\rm T }(e_1) & \dots &  & & sa_1{\rm T }(e_1) \\
    sa_2{\rm T }( e_2) & r+sa_2{\rm T }( e_2) & \dots & & & sa_2{\rm T }(e_2) \\
    \vdots & &  \ddots &  & & \vdots  \\
        &  & &  & & sa_{N-1}{\rm T }(e_{N-1}) \\
      sa_N{\rm T }( e_N) &   s a_N{\rm T }(e_N)  & \dots  & &  & r+s a_N{\rm T }( e_N)
    \end{bmatrix}.
\] Notice that the sum of the elements of an arbitrary column of $A$ is equal to \[r+s\sum_{i=1}^{N}a_{i}{\rm T}(e_{i})=r+s{\rm T}(\v_{1})=m.\] Therefore after adding all the rows of $A$, factorizing  $m$, and then substracting from the $i$-th row $sa_{i}{\rm T}(e_{i})$ times the first row, we see that $\det(A)=mr^{N-1}.$ To finish the proof we must prove that such a rigid basis actually exists. Since ${\rm T}:\L\rightarrow \Z$ is non-trivial there is $v \in \L$ such that ${\rm Im}({\rm T})$ is generated by ${\rm T}(v)$, moreover by the additivity of the rank, $\ker {\rm T}$ is a sublattice of $\L$ of rank $n-1$. Let $\{w_1,\dots, w_{n-1}\}$ be a basis for $\ker {\rm T}$. Notice that $\{w_1,\dots,w_{n-1}, v\}$ is a basis for $\L$. Thus,  $\mathcal{B}:=\{w_1+v,\dots,w_{n-1}+v, v\}$  is a basis for $\L$ and it is rigid with respect to ${\rm T}$ since ${\rm T}$ takes the value of ${\rm T}(v)$ at every element in the set.
\end{proof}

\begin{example}\label{CubicLatticeFirstEx}
Suppose $N \ge 2$ and let $\L =\Z^{N}$ be the standard cubic lattice. Let ${\rm T}: \L \to  \Z$ be the map that sends a vector to the sum of its coordinates. In this case the lattice $\L_{{\rm T}}^{(2)}$ is the lattice $\mathbb{D}_{N}$. 
\begin{enumerate}

\item Choosing $\v_{1}$ to be any vector with ${\rm T}(\v_{1})=3$, for instance $\v_{1}=[1,1,1,0...,0]^{t}$ if $N>2$ and  $\v_{1}=[1,2]^{t}$ for $N=2$, we obtain that  $\L_{ {\rm T},\v_1}^{(-1,1)}=\L_{{\rm T}}^{(2)}=\mathbb{D}_{N}$. This can be seen by calculating the image of the standard basis of $\Z^{N}$ under $\Phi_{(r,s)}$. Such image is the set \[\{[0,1,1,0,...,0]^t, [1,0,1,0,...,0]^t, [1,1,0,0,...,0]^t , [1,1,1,-1,...,0]^t,...,[1,1,1,0,...0,-1]^t \}\] for $N >2$ and $\{[0,2]^{t}, [1,1]^{t}\}$ for $N=2$. In either case such set is a basis of $\mathbb{D}_{N}$.

\item Suppose that $N>2$. Choosing  $\v_{1}$ as the vector $[1,...,1]^{t}$, we see that for $m=2=r+sN$  the tuple $(r,s)$ is either $(2,0)$ or $(2-N, 1)$ which define the lattices   $\L_{ {\rm T},\v_1}^{(2,0)}$ and $\L_{ {\rm T},\v_1}^{(2-N,1)}$ respectively. In the former case, the lattice $\L_{ {\rm T},\v_1}^{(2,0)}$ is $(2\Z)^{N}$. In the latter, the structure of the lattice $\L_{ {\rm T},\v_1}^{(2-N,1)}$ varies with $N$. Similar to the previous item a basis for $\L_{ {\rm T},\v_1}^{(2-N,1)}$  can be calculated as the image of the standard basis of $\Z^{N}$ under $\Phi_{(r,s)}$ to obtain the basis  \[\{[3-N,1,1,1,...,1]^t, [1,3-N,1,1,...,1]^t, ...,[1,1,...,1,3-N]^t \}.\] 

 For instance:
\begin{itemize}

 \item If $N=3 $ the lattice $\L_{ {\rm T},\v_1}^{(-1,1)}$ is isomorphic to the root lattice $\mathbb{A}_{3}$. 
 
 \item If $N=4$ the lattice $\L_{ {\rm T},\v_1}^{(-2,1)}$ is isomorphic to the lattice $(2\Z)^{4}$ .
 
 \item If $N=5$ the lattice $\L_{ {\rm T},\v_1}^{(-3,1)}$ is isomorphic to the lattice $L_{9,5}$ defined in \cite[Theorem 4.1]{BaNe}.
 
 \end{itemize}

\item For $N=2$ and $\v_{1}=[1,1]^t$ the lattices $\L_{ {\rm T},\v_1}^{(-1,2)} , \L_{ {\rm T},\v_1}^{(1,1)}$, and $\L_{{\rm T}}^{(3)}$ are all isomorphic to $\mathbb{A}_{2}$.  

\end{enumerate}
\end{example}

\begin{lemma}\label{Cong1}
Let $N \ge 2$ be an integer and let $K$ be a degree $N$ number field. Let $m$ be an integer such that  $m\equiv \pm 1 \pmod{N}$. By the Minkowski embedding we can view $\L:=O_{K}$ as a lattice in $\R^{N}$, so for this lattice let  ${\rm T}: \L \to \Z$ be the trace map. Suppose that $K$ is tame, i.e., that there is no prime that ramifies wildly in $K$. Then,  \[\{x \in O_{K}: {\rm Tr}_{K/\Q}(x) \equiv 0\pmod{m}\}=  \L_{T,1}^{\pm 1, s} \] where $s=\frac{m \mp 1}{N}.$
\end{lemma}

\begin{proof} Since $K$ is tame the trace map ${\rm T}: \L \to \Z$ is surjective, see \cite[Corollary 5 to Theorem 4.24]{narkiewicz2004elementary}. Hence the result follows from Proposition \ref{Index} and Lemma \ref{CongruencemLattice}.

\end{proof}

\subsection{Construction }In this section we specialize the above construction to lattices that have properties that are motivated by the structure of the ring of integers of certain tame abelian totally real number fields.  Let $\langle \cdot , \cdot \rangle$ be the standard inner product in $\R^{N}$. Suppose that $\L$ is a rank $N$ lattice such that  $\L \cap \L^{*} \neq 0$. This, for instance, can be achieved if $\L$ is integral. By definition, any non-zero $\v_{1} \in \L \cap \L^{*} $ defines an element  ${\rm T}_{\v_{1}} \in {\rm Hom}_{\Z}(\L, \Z) \setminus \{0\} $; namely \[{\rm T}_{\v_{1}} (x)=\langle x, \v_{1} \rangle.\]

 \begin{definition} Let $\L$ and $\v_{1}$ as above. Let $r, s$ be integers with  $0 \neq |r| < {\rm T}(\v_{1})= \| \v_{1} \|^2$.  The lattice  $\L_{\v_1}^{(r,s)}$ is defined as \[\L_{\v_1}^{(r,s)}:=\L_{ {\rm T}_{\v_{1}},\v_1}^{(r,s)}.\]
\end{definition}

From now on $\v_{1}$ will be a fixed non-zero element in  $\L \cap \L^{*}$ and, unless clarification is necessary, we will denote the map  ${\rm T}_{\v_{1}}$  by ${\rm T}$.

\begin{proposition}\label{prop1}
 Let $\L$, $\v_{1}$, $r$, and $s$ be as above. Let $m:=r+s {\rm T}(\v_{1})=r+s \|\v_{1}\|^{2}.$  For all $\alpha \in \L$ we have

\[||\Phi_{(r,s)}(\alpha)||^2 = A ||\alpha||^2 + B {\rm T}^2(\alpha),\]
where $\displaystyle A=r^2$ and $\displaystyle B=\frac{m^2-r^2}{\| \v_{1}\|^2}$.
\end{proposition}
\begin{proof} Let $\alpha \in \L$ and recall that by definition $\Phi_{(r,s)}(\alpha) =r\alpha +s{\rm T}(\alpha) \v_{1}$.  Then, 
\begin{equation*} 
\begin{split}
||\Phi_{(r,s)}(\alpha)||^2 & = \left \langle \Phi_{(r,s)}(\alpha),\Phi_{(r,s)}(\alpha) \right\rangle \\
 & = \left \langle r\alpha +s{\rm T}(\alpha) \v_{1} ,  r\alpha +s{\rm T}(\alpha) \v_{1} \right\rangle \\ 
 & = r^2||\alpha||^2 + 2rs{\rm T}(\alpha) \langle \alpha, \v_{1} \rangle + s^2{\rm T}(\alpha)^{2}  \|\v_{1}\|^{2} \\
 & = r^2||\alpha||^2 + 2rs{\rm T}(\alpha)^{2}  + s^2{\rm T}(\alpha)^{2}  \|\v_{1}\|^{2} \\
 & = r^2||\alpha||^2 + \left(2rs+s^2\|\v_{1}\|^{2}\right){\rm T}(\alpha)^2\\
 & = r^2||\alpha||^2 + s(r+m){\rm T}(\alpha)^2\\
 &= r^2||\alpha||^2 +  \frac{(m^2-r^2)}{\|\v_{1}\|^{2}}{\rm T}(\alpha)^2
\end{split}
\end{equation*}
\end{proof}

\subsection{Tame lattices}\label{TameLattices}

In this section we define the notion of tame lattice, then we use this definition to construct lattices with a minimal basis. 
\begin{example}
Suppose $N \ge 2$ and let $\L =\Z^{N}$ be the standard cubic lattice. Let $\v_{1}:=[1,...,1]^{t}$. For this choice of $\v_{1}$ the map ${\rm T}$ is equal to the sum of entries of a vector in $\Z^{N}$.  As we have seen in example \ref{CubicLatticeFirstEx}, for any pair of integers $(r,s)$ with $0< |r| <N$ the lattice $\L_{\v_1}^{(r,s)}$ can be very interesting and diverse. For instance, for $N>2$, by just picking $r \equiv 2 \pmod N$ we obtained $\mathbb{D}_{N}$, $\mathbb{A}_{2}$, $(2\Z)^{4}$, and $L_{9,5}$ (see \cite[Theorem 4.1]{BaNe} for its definition).
\end{example}

Another important feature of these examples is that all of them are strongly well-rounded. 

\begin{definition}\label{Lagrangian} Let $N$ be a positive integer and let $\L$ be a rank $N$ lattice. We say that $\L$ is tame if 
there is a basis $\{e_{1},...,e_{N}\}$ of $\L$ and a non-zero $\v_{1} \in \L \cap \L^{*}$ such that 

\begin{enumerate}

\item $e_{1}+...+e_{N}=\v_{1}.$

\item ${\rm T}_{\v_{1}}(e_{i})=\langle e_{1}, \v_{1} \rangle = 1$ for all $1 \leq i \leq N$.

\item $\langle e_{i}, e_{i} \rangle = \langle e_{j}, e_{j} \rangle$ for all $1 \leq i, j \leq N$.

\item $ \langle e_{i}, e_{j} \rangle = \langle e_{k}, e_{l} \rangle$ for all $1 \leq i, j, k,l \leq N$ with $i\neq j$ and $k \neq l$.

\end{enumerate}

In such a case, the basis $\{e_{1},...,e_{N}\}$ is called a Lagrangian basis.
\end{definition}

\begin{remark}
 Sets of vectors $v_1,\dots,v_k$ in $\R^N$ satisfying $\langle v_{i}, v_{i}\rangle=1$ and $\alpha=|\langle v_{i}, v_{j}\rangle|$ for $i\neq j$, and some $\alpha\in \R$, are sometimes called equiangular unit frames (see \cite{sustik2007existence}).
 \end{remark}

\begin{lemma}\label{NormaV1}
 Let $\L$ be a rank $N$ tame lattice and let $\v_{1}$ be the vector defining the map ${\rm T}$. Then, \[ \|\v_{1}\|^{2}={\rm T}(\v_{1})=N.  \]
\end{lemma}
\begin{proof}
By definition of {\rm T}, we have that$ \|\v_{1}\|^{2} =\langle \v_{1}, \v_{1} \rangle= {\rm T}(\v_{1})$.  On the other hand let $\{e_{1},...,e_{N}\}$ be a basis of $\L$ satisfying the conditions of Definition \ref{Lagrangian}.  Then, ${\rm T}(\v_{1})={\rm T}(e_{1}+...+e_{N})={\rm T}(e_{1})+...+{\rm T}(e_{N})=1+...+1=N.$
\end{proof}

\begin{proposition}\label{prop2} Let $\L$ be a rank $N$ tame lattice with $\v_{1} \in \L$ and $\{e_{1},...,e_{N}\}$ satisfying the conditions of Definition \ref{Lagrangian}. Let $a:=\langle e_{1}, e_{1}\rangle$ and $h=-\langle e_{1}, e_{2}\rangle.$  Let $r, s$ be integers such that  $0 \neq |r| <N$ and let $m:=r+sN.$

\begin{itemize}
    \item $\displaystyle \L_{\v_1}^{(r,s)}$ is a sub-lattice of $\L$ of index $m|r|^{N-1}$.
    \item $\displaystyle ||\Phi_{(r,s)}(e_i)||^2 = ar^2 +\frac{m^2-r^2}{N}$ for all $1  \leq i \leq N$.
    \item $\displaystyle \langle \Phi_{(r,s)}(e_i),\Phi_{(r,s)}(e_j) \rangle= -r^2 h +\frac{m^2-r^2}{N}$ for all $1  \leq i \neq j\leq N$.
\end{itemize}

\end{proposition}
 
 \begin{proof}
 The first two conditions follow from  Propositions \ref{Index},  \ref{prop1} and Lemma \ref{NormaV1}. To finish the proof we take  $i\neq j$. Then, \[ \langle \Phi_{(r,s)}(e_i),\Phi_{(r,s)}(e_j) \rangle= \langle re_{i} +s \v_{1} ,  re_{j} +s\v_{1} \rangle=  r^2\langle e_{i}  ,  e_{j} \rangle +2rs+s^2N=-r^2 h +\frac{m^2-r^2}{N}.\]
 \end{proof}

 \begin{definition}
  Let $\L$ be a rank $N$ lattice and let  ${\rm T}: \L \to \Z$ be a non-trivial linear map. We denote by $\L_{\rm T}^{0}$ the rank $N-1$ lattice given by de Kernel of ${\rm T}$;   $\L_{\rm T}^{0}:=\ker({\rm T})$. We  will denote this sub-lattice by $\L^{0}$ since in general ${\rm T}$ will be clear from the context.
 \end{definition}
 
 \begin{example}
 Taking $\L=\Z^{n+1}$ the usual cubic lattice and ${\rm T}$ the sum of the coordinates function, then $\L^{0}$ is the root lattice $\mathbb{A}_{n}$. 
 \end{example}
 
 The above example is just a particular case of what happens in general tame lattices.
 
\begin{theorem}\label{An}
Let $\L$ be a rank $N$ tame lattice with $\v_{1} \in \L$ and $\{e_{1},...,e_{N}\}$ satisfying the conditions of Definition \ref{Lagrangian}. Let $a:=\langle e_{1}, e_{1}\rangle$ and $h=-\langle e_{1}, e_{2}\rangle.$   Then, \[\L^{0} \cong (a+h)\mathbb{A}_{N-1}.\]
\end{theorem}

\begin{proof}
Let $w_{1}:=e_{1}-e_{2}, w_{1}:=e_{2}-e_{3},...,w_{N-1}:=e_{N-1}-e_{N}$. Since $\rm T$ is constant on the $e_{i}'s$ then $w_{i} \in \L^{0}$ for all $i$. We claim that $\{w_{1},...,w_{N-1}\}$ is a basis of  $\L^{0}$. They are clearly linearly independent, so it's enough to show that ${\rm spam}_{\Z} \{w_{1},...,w_{N-1}\}= L^{0}$. Let $v =a_{1}e_{1}+...+a_{N}e_{N} \in L^{0}.$ Then $a_{1}+...+a_{N}=0$. To show that $v \in {\rm spam}_{\Z} \{w_{1},...,w_{N-1}\}$ is equivalent to show that the following linear system is solvable in $\Z$ 
\begin{equation*}
\begin{split}
b_{1} & = a_{1} \\
b_{2} -b_{1} & = a_{2}\\
b_{3} -b_{2} & = a_{3}\\
\vdots \ \ \ & = \ \   \vdots\\
b_{N-1} -b_{N-2} & = a_{N-1}\\
-b_{N-1} & = a_{N}
\end{split}
\end{equation*} Since $b_{1}+ (b_{2} -b_{1})+ (b_{3} -b_{2}) +...+(b_{N-1} -b_{N-2})=b_{N-1}$ and $a_{1}+...+a_{N-1}=-a_{N}$ the system above has a solution if and only if the system given by the first $N-1$ equations has a solution, and this last system is clearly solvable.  Returning to the main proof notice that for all $1 \leq i \leq N-1$, \[\langle w_{i}, w_{i} \rangle=\langle e_{i}, e_{i} \rangle-2\langle e_{i}, e_{i+1} \rangle+\langle e_{i+1}, e_{i+1} \rangle=2(a+h).\] Suppose that $1 \leq i<j \leq N-1$. If $j\neq i+1$, then \[\langle w_{i}, w_{j} \rangle=\langle e_{i}, e_{j} \rangle-\langle e_{i}, e_{j+1} \rangle-\langle e_{j}, e_{i+1} \rangle+\langle e_{j+1}, e_{j+1} \rangle=-h+h+h-h=0.\] For $j =i+1$, \[\langle w_{i}, w_{i+1} \rangle=\langle e_{i}, e_{i+1} \rangle-\langle e_{i}, e_{i+2} \rangle-\langle e_{i+1}, e_{i+1} \rangle+\langle e_{i+1}, e_{i+2} \rangle=-h+h-a-h=-(a+h).\] Therefore the Gram matrix of $\L^{0}$ in the basis $\{w_{1},...,w_{N-1}\}$ is $(a+h)M$ where $M$ is one of the known Gram matrices of $\mathbb{A}_{N-1}$ (See \cite{conway2013sphere}).
\end{proof}

\begin{corollary}\label{TraceZero}
Let $\L$ be a rank $N$ tame lattice with $\v_{1} \in \L$ and $\{e_{1},...,e_{N}\}$ satisfying the conditions of Definition \ref{Lagrangian}. Let $a:=\langle e_{1}, e_{1}\rangle$ and $h=-\langle e_{1}, e_{2}\rangle.$ Then \[\min_{v \in \L^{0}\setminus \{0\} } \| v\|^{2}=2(a+h).\] Furthermore, such minimum is obtained at any of the vectors $e_{i}-e_{j}$ for $i\neq j$
\end{corollary}

\begin{proof}
This follows from Theorem \ref{An} and from the fact that for any integer $n>1$ the root lattice $\mathbb{A}_{n}$ has minimal distance equal to $2$. 
\end{proof} 
 
\section{Finding a minimal basis}\label{Condi}

As we have seen in previous examples, by considering the construction $\L^{(r,s)}_{\v_{1}}$ applied to the tame lattice $\Z^{N},$ for some values $(r,s)$, we obtained several well-rounded lattices, moreover they have a minimal basis.
Here, we show how to do this for an arbitrary tame lattice $\L$ provided that we have some restrictions on the values $(r,s)$ with respect to $\L$.\\

Throughout this section  $\L \subseteq \R^{N}$ will denote a tame lattice with $\v_{1}$ and $\{e_{1},...,e_{N}\}$ satisfying the conditions of Definition \ref{Lagrangian}. Let $a:=\langle e_{1}, e_{1}\rangle$ and $h=-\langle e_{1}, e_{2}\rangle.$  Let $r, s$ be integers such that  $0 \neq |r| <N$ and let $m:=r+sN.$  Also recall the definitions  $A=r^2$ and $\displaystyle B=\frac{m^2-r^2}{N}$.


\subsection{Shortest vector problem for $\L_{\v_{1}}^{(r,s)}$.}

The shortest non-zero norm in $\L_{\v_{1}}^{(r,s)}$ is by definition  \[\lambda_{1}(\L_{\v_{1}}^{(r,s)}):=\min_{v \in \L_{\v_{1}}^{(r,s)} \setminus \{0\} } \| v\|^{2}.\]

Thanks to Proposition \ref{prop1} and Lemma \ref{NormaV1} we have that  \[\lambda_{1}(\L_{\v_{1}}^{(r,s)}):=\min_{\alpha \in \L \setminus \{0\} } \left(A ||\alpha||^2 + B {\rm T}^2(\alpha)\right).\]

Therefore we are left with the  task of minimizing the function \[f(\alpha):=A ||\alpha||^2 + B {\rm T}^2(\alpha)\] on $\L \setminus \{0\}$. To do this we use the natural partition of $\L$ given by taking the quotient with the sub-lattice $\L^{0}$. Since $\L/\L^{0} \cong \Z$ via the map {\rm T} such partition is  \[\L = \bigcup_{d\in\Z}S_d \] where $S_d:= \{x\in \L  \ | \  {\rm T}(x)= d\}.$ Hence we have  \[\lambda_{1}(\L_{\v_{1}}^{(r,s)}):=\min_{d \ge 0 } (\min_{\alpha \in S_{d} \setminus \{0\}} f(\alpha)).\]

\begin{remark}
We included only non-negative values of $d$ above since $f$ is even and $S_{-d}=-S_{d}$.
\end{remark}

\begin{lemma}\label{Restar}

Let $d$ be an integer. Let $\alpha=a_{1}e_{1}+...+a_{N}e_{N} $ be an element of $S_{d}$.  Suppose there are $a_{i}, a_{j}$ such that $a_{i}-a_{j} > 1$. Then there is $\beta \in S_{d}$ such that $\| \beta\| < \| \alpha \|.$

\end{lemma}

\begin{proof}

Let $\beta= \alpha + e_{j} -e_{i}.$ Since $e_{j}-e_{i} \in \ker({\rm T})$ we have that $\beta \in S_{d}$. Taking square norms we get \[\|\beta\|^{2}=\|\alpha\|^{2}+ 2\left(\langle \alpha, e_{j}\rangle- \langle \alpha, e_{i}\rangle\right) +  \|e_{i}\|^{2}- 2\langle e_{i}, e_{j}\rangle+\|e_{j}\|^{2}=\|\alpha\|^{2}+ 2\left(\langle \alpha, e_{j}\rangle- \langle \alpha, e_{i}\rangle +a+h \right).\] On the other hand, for all  $1 \leq k \leq N$, \[ \langle \alpha, e_{k}\rangle= \sum_{l \neq k} a_{l}\langle e_{l}, e_{k}\rangle+ a_{k}a=-h\sum_{l \neq k} a_{l}+a_{k}a=-h({\rm T}(\alpha) -a_{k} )+ a_{k}a=-h{\rm T}(\alpha) + a_{k}(a+h).\] In particular, $\langle \alpha, e_{j}\rangle- \langle \alpha, e_{i}\rangle = (a_{j}-a_{i})(a+h)$ and thus \[\|\beta\|^{2}=\|\alpha\|^{2}+2(a+h)(a_{j}-a_{1}+1) =\|\alpha\|^{2}-2(a+h)(a_{i}-a_{j}-1) < \|\alpha\|^{2}.\]

\end{proof}

For a subset $I \subset \{1,...,N\}$ we denote by $\displaystyle E_{I}:=\sum_{i \in I} e_{i}$. For instance, $E_{\emptyset}=0$ and $E_{\{1,...,N\}}=\v_{1}$. Notice that  ${\rm T}(E_{I})=\#I$ for all subset $I$.

\begin{corollary}\label{minlemma}
Let $d$ be a positive integer. Then $\displaystyle \min_{\substack{x\in S_d\\d\neq 0}}\{||x||^2\}$ is achieved at some \[\alpha=E_I +C\v_{1}\]
where $C$ is a non-negative integer and $\#I < N$. Moreover, \#I is the residue class of $d$ modulo $N$ and $C=(d-\#I)/N.$
\end{corollary}
\begin{proof}
Let $\displaystyle \alpha=\sum_{i=1}^{N}a_i e_i\in S_d$ a non-zero element having minimal norm.  Let \[a_k=\max\{a_i~:~1\leq i\leq N\}.\] Replacing $\alpha$ by $-\alpha$ if necessary, we may assume that $a_k \ge 1$.  Suppose that there is some $a_{j} \leq 0$. Then, since the $a_{j}$s can not be more than one integer apart, by Lemma \ref{Restar}, we must have that all non-positive coefficients are equal to $0$ and all positives are equal to $1$. Thus in such case $\alpha$ is of the form $E_{I}$ where $\#I<N$ since there are coefficients equal to $0$. Now, if all $a_{j}$s are positive let \[c=\min\{a_i~:~1\leq i\leq N\}.\] If $c=a_{k}$ then $\alpha =c\v_{1}$, otherwise $a_{k}=c+1$ and 
$\alpha=E_{I}+c\v_{1}$ where $I$ is the subset of elements $i$ such that $a_{i}=a_{k}$. 
\end{proof}

\begin{theorem}\label{MinPartd}
Let $d$ be a positive integer. Let $k$ be the residue class of $d$ modulo $N$ and let $c=(d-k)/N$. Then, \[\min_{\alpha \in S_{d} \setminus \{0\}} f(\alpha)=f(E_{I})+c^2f(\v_{1})+ 2ck(A+NB)\] where $I$ is a subset of size $k$.

\end{theorem}

\begin{proof}

Since in $S_{d}$ the function ${\rm T}$ is constant, the minimum value of $f(\alpha)=A\|\alpha\|^{2}+B {\rm T}(\alpha)^{2}$ is attained whenever $\|\alpha\|^{2}$ is minimal. Thanks to Corollary \ref{minlemma} such a minimum is attained at $\alpha=E_I +c\v_{1}$, hence the minimum value of $f$ over $S_{d}$ is $f(E_I +c\v_{1})=f(E_{I})+c^2f(\v_{1})+ 2ck(A+NB).$

\end{proof}

\begin{lemma}\label{lemmaB}
Let $k$ be an integer $0\leq k\leq N$ and let $I\subset\{1,\dots,N\}$ such that $|I|=k$, then
\[||E_I||^2=k(a-(k-1)h)=k(1+(N-k)h).\] In particular, \[f(E_{I})= Ak(1+(N-k)h)+Bk^2.\]

\end{lemma}
\begin{proof}
We may assume that $k \neq 0$. \[\||E_I||^2 =\langle E_{I}, E_{I} \rangle =\sum_{i\in I} \langle e_{i}, e_{j} \rangle +  \sum_{\substack{ {i,j\in I} \\ i\neq j}} \langle e_{i},  e_{j}\rangle =\sum_{i\in I} a  -  \sum_{\substack{ {i,j\in I} \\ i\neq j}}h=ak-(k^2-k)h=k(a-(k-1)h).\] Using the case $k=N$, i.e., $E_{I}=\v_{1}$ we see that $N=\|E_I||^2= N(a-(N-1)h)$ hence $a=1+(N-1)h$. Replacing this in the first equality the second follows, and so it does the claim about $f(E_{I})$.
\end{proof}

\subsection{Conditions on minimal basis.}

To see whether or not $\L^{(r,s)}_{\v_{1}}$ has a minimal basis, firstly we should find a basis in which all the vectors have the same norm. Since we have calculated the min values of $f$ over each $S_{d}$, we then should compare the value of the norms of the proposed basis versus the minimal values of $f$. Thanks to Proposition \ref{prop2} we have that all the elements of the basis $\{\Phi_{r,s}(e_{1}), ..., \Phi_{r,s}(e_{N})\}$ have norm $aA+B$. Hence, if $aA+B$ happened to be equal to $\lambda_{1}(\L_{\v_{1}}^{(r,s)})$, we should have at least $aA+B \leq f(\v_{1}) =N(A+NB).$ As it turns out this is already a pretty strong condition as the next theorem shows.

\begin{theorem}\label{AlmostMain}

Suppose that $aA+B \leq f(\v_{1}).$ Then,  \[aA+B =\min_{d > 0 } (\min_{\alpha \in S_{d} \setminus \{0\}} f(\alpha)).\]

\end{theorem} 

\begin{proof}
Since $e_{1} \in  S_{1}$ and $aA+b=f(e_{1})$ we have that $\displaystyle \min_{d > 0 } (\min_{\alpha \in S_{d} \setminus \{0\}} f(\alpha)) \leq aA+B$. To show the opposite inequality; consider $d$ to be a positive integer, and let $k$ be the residue class of $d$ modulo $N$ and $c=(d-k)/N$.  Thanks to Theorem \ref{MinPartd} we have that \[\min_{\alpha \in S_{d} \setminus \{0\}} f(\alpha)=f(E_{I})+c^2f(\v_{1})+ 2ck(A+NB)\] where $I$ is subset of $\{1,...,N\}$ of size $k$. If $c\neq 0$, we have that $f(E_{I})+c^2f(\v_{1})+ 2ck(A+NB) \ge f(\v_{1}) \ge aA+B$. If $c=0$, then $k \neq 0$ and thanks to the next proposition $f(E_{I}) \ge aA+B$. Thus, in either case  \[\min_{\alpha \in S_{d} \setminus \{0\}} f(\alpha) \ge aA+B\] from which the result follows.

\end{proof}

\begin{proposition}\label{prop6}
Let $1 \leq k\leq N$ be an integer, and let $I\subset\{1,\dots,N\}$ be a subset of size $k$. Suppose that $aA+B \leq f(\v_{1})$. Then,
\[ aA+B \leq  f(E_{I}).\]
\end{proposition}
\begin{proof}

Consider the parabola \[g(t):=At(1+(N-t)h)+Bt^2.\]

By Lemma \ref{lemmaB} we have that $aA+B=g(1), f(\v_{1})=g(N)$, and $f(E_{I})=g(k)$. Also, notice that $g(0)=0$. Thus, $ g(0) < g(1) \leq g(N)$. Since $g$ is a parabola and $0<1<N$, the function $g$ is increasing in the interval $[1,N]$. Hence, $aA+B=g(1) \leq g(k)=f(E_{I})$.

\end{proof}

\begin{corollary}\label{ElUltimoCoro}
Suppose that $aA+B \leq N(A+NB).$ Then, \[ \lambda_{1}(\L_{\v_{1}}^{(r,s)})=\min\{ 2A(a+h), aA+B \}.\]
\end{corollary}

\begin{proof}
Recall that by considering the quotient partition of $\L/\L^{0}$, we have that \[\lambda_{1}(\L_{\v_{1}}^{(r,s)})=\min_{d \ge 0 } (\min_{\alpha \in S_{d} \setminus \{0\}} f(\alpha)).\] Therefore, thanks to Theorem \ref{AlmostMain}, $\displaystyle \lambda_{1}(\L_{\v_{1}}^{(r,s)})=\min\{ \min_{\alpha \in S_{0} \setminus \{0\}} f(\alpha), aA+B \}$. On the other hand, $f(\alpha)=A\| \alpha\|^{2}$ for $\alpha \in S_{0}=\L^{0}$. The result follows from Corollary \ref{TraceZero}.
\end{proof}
 
 We are ready to summarize our results in the main theorem of the paper:
 
\begin{theorem}\label{main}

Let $\L \subseteq \R^{N}$ be a tame lattice with $\v_{1}$ and $\{e_{1},...,e_{N}\}$ satisfying the conditions of Definition \ref{Lagrangian}. Let $a:=\langle e_{1}, e_{1}\rangle$ and $h=-\langle e_{1}, e_{2}\rangle.$  Let $r, s$ be integers such that  $0 \neq |r| <N$ and let $m:=r+sN.$ Suppose that 
\[\frac{Na-1}{N^2-1}\leq \left(\frac{m}{r}\right)^2\leq \frac{(aN-1)(N+1)}{N-1}.\]

Then the lattice $\L_{\v_1}^{(r,s)}$ is a sub-lattice of $\L$ of index $m|r|^{N-1}$, with minimum \[\lambda_1(\L_{\v_1}^{(r,s)})=ar^2+\frac{m^2-r^2}{N}\] and with basis of minimal vectors \[\{re_1+s\v_1,re_2+s\v_1,\dots,re_N+s\v_1\}.\]
\end{theorem}
\begin{proof}
 By Corollary \ref{ElUltimoCoro}, the lattice $\L_{\v_1}^{(r,s)}$ has a minimal basis with minimal norm $\displaystyle aA+B$ if and only if 
 \begin{itemize}
 
 \item $Aa+B \leq N(A+NB)$.
 
 \item $Aa+ B \leq 2A(a+h).$
 
 \end{itemize}

Recall that $A=r^2$ and $B=\frac{m^2-r^2}{N}$. Thus, using that $\frac{B}{A}= \frac{1}{N} \left(\left(\frac{m}{r}\right)^2 -1\right)$ the inequality \[ Aa+B \leq N(A+NB) \ \mbox{turns into} \ \displaystyle \frac{Na-1}{N^2-1}\leq \left(\frac{m}{r}\right)^2 \]and using that $a=1+(N-1)h$, see the proof of Lemma \ref{lemmaB},  \[ Aa+ B \leq 2A(a+h)  \ \mbox{turns into } \ \left(\frac{m}{r}\right)^2\leq \frac{(aN-1)(N+1)}{N-1}.\]

\end{proof}

\begin{corollary}\label{LagragianNumberFieldCoro}
Let $N \ge 2$ be an integer. Let $K$ be a tame totally real degree $N$ number field. Let $m$ be an integer such that  $m\equiv \pm 1 \pmod{N}$. Suppose that $O_{K}$ has an integral Lagrangian basis $\{e_{1},...,e_{N}\}$ such that $1=e_{1}+e_{2}+...+e_{N}$. Let $a:=\langle e_{1}, e_{1}\rangle$ and $h=-\langle e_{1}, e_{2}\rangle$, and suppose that 
\[\frac{Na-1}{N^2-1}\leq m^2\leq \frac{(aN-1)(N+1)}{N-1}.\]
Then, the lattice
\[\{x \in O_{K}: {\rm Tr}_{K/\Q}(x) \equiv 0\pmod{m}\} \] is a sub-lattice of $O_{K}$ that has a minimal basis with minimum $\lambda_{1}=a+\frac{m^2-1}{N}.$ 

\end{corollary}

\begin{proof}
The result follows immediately from Lemma \ref{Cong1} and Theorem \ref{main}.
\end{proof}

\begin{remark}
Note that for simplicity we have only mentioned the case $m\equiv \pm 1 \pmod{N}$. However, the more general values of $m=r+sN$ yield lattices with minimal bases inside $\{x \in O_{K}: {\rm Tr}_{K/\Q}(x) \equiv 0\pmod{m}\}.$

\end{remark}

\subsubsection{Some number fields with Lagrangian basis}\label{NumberFieldsWithLagrangianBasis}

There are several examples of real number fields containing an integral Lagrangian basis. One of the first families of such number fields was found in the mid 80s by Conner and Perlis while studying integral traces over tame Galois number fields of prime degree:

\begin{example}\label{ex1}
Let $p$ be a prime and let $K$ be a Galois number field of degree $p$. If $K$ is tame, which in this case is equivalent to say that $p$ does not ramify, then $O_{K}$ has an integral Lagrangian basis $\{e_{1},...,e_{p}\}$ with $a=\frac{n(p-1)+1}{p}$ and $h=\frac{n-1}{p}$  where $n$ is the conductor of $K$ (see \cite{CoPe}). Note that such values of $a$ and $h$ for for $N=p$ we have that $\frac{Na-1}{N^2-1}=\frac{n}{p+1}$ and $\frac{(aN-1)(N+1)}{N-1}=n(p+1)$. Therefore, if $m \equiv \pm 1\pmod{p}$ is such that $\frac{n}{p+1} \leq  m^2 \leq n(p+1) $, then thanks to Corollary \ref{LagragianNumberFieldCoro} \[\{x \in O_{K}: {\rm Tr}_{K/\Q}(x) \equiv 0\pmod{m}\}\] is a  sub-lattice of $O_{K}$ with a minimal basis and minimum $\lambda_{1}=\frac{n(p-1)+m^2}{p}.$ When restricting this example to the case $m \equiv 1\pmod{p}$ the results [Theorem 4.1,\cite{sueli}] and [Theorem 3.3,\cite{oliviera}] are recovered. 
\end{example}

Recently the results of Conner and Perlis about Trace forms over cyclic number fields, see \cite{BoMa}, have been generalized. Using these ideas, the example above can be extended to number fields of not necessarily  prime degree:

\begin{example}\label{ex2}
Let $K$ be a tame totally real abelian number field of degree $N$. If the conductor $n$ of $K$ is prime, then $O_{K}$ has an integral Lagrangian basis $\{e_{1},...,e_{N}\}$ with $a=\frac{n(N-1)+1}{N}$ and $h=\frac{n-1}{N}$. (The construction of such a basis can be done as in \cite[Lemma 3.3]{BoMa}; there, this basis is constructed for $N$ a prime power. However, the same exact proof works for any $N$.) Thus, as in the previous example our general construction can also be applied to such fields.  For instance, the field $K=\Q(\zeta_{13}+\zeta_{13}^{-1})$  is a tame real Galois extension of $\Q$ with Galois group $\Z/2\Z \times \Z/3\Z$ and of conductor $13$. The lattices $\{x \in O_{K}: {\rm Tr}_{K/\Q}(x) \equiv 0\pmod{5}\}$ and $\{x \in O_{K}: {\rm Tr}_{K/\Q}(x) \equiv 0\pmod{7}\}$ are sub-lattices of $O_{K}$ with minimal bases, and with respective minima equal to $15=\frac{13(6-1)+5^2}{6}$ and $19=\frac{13(6-1)+7^2}{6}$.
\end{example}

The two families in the above examples are not the only examples of number fields with Lagrangian integral basis. The following example shows that there exist abelian fields that are neither of prime degree or prime conductor and that have a Lagrangian integral basis.

\begin{example}\label{ex3}

Let $K$ be the number field defined by the polynomial $f:=x^4 - x^3 - 24x^2 + 4x + 16$. The field $K$  is $\Z/4\Z$-extension of $\Q$, has discriminant $5^3\cdot 13^3$, and conductor $n=65$. If $\{a_{1},a_{2}, a_{3}, a_{4}\}$ is the set of roots of $f$, then they form an integral basis of $O_{K}$, $a_{1}+a_{2}+a_{3}+a_{4}=1$, and the Gram matrix of the trace in such a basis is 
\[\begin{bmatrix} \ \ 49 & -16 & -16 & -16 \\ -16 & \ \ 49 & -16 & -16  \\ -16 & -16 & \ \ 49 & -16 \\ -16 & -16 & -16 &  \ \ 49  \end{bmatrix}.\] 
Hence, $O_{K}$ is a tame lattice with $a=49, h=16$, and $N=4$. Applying the bounds of Corollary \ref{LagragianNumberFieldCoro} here, we  obtain 
 $13 \leq m^2 \leq \frac{(aN-1)(N+1)}{N-1}=325$ for $m \equiv \pm 1 \pmod{4}$. This is equivalent to $m \in \{5,7,9,11,13,15,17\}.$ Hence, by applying Corollary \ref{LagragianNumberFieldCoro} to $K$ which is neither a prime degree nor a prime conductor number field, seven non-isometric sub-lattices of $O_{K}$ all with a minimal basis have been constructed.
\end{example}

\bibliographystyle{abbrv}
\bibliography{ref}

\begin{thebibliography}{10}

\bibitem{ash}
A.~Ash et~al.
\newblock Small-dimensional classifying spaces for arithmetic subgroups of
  general linear groups.
\newblock {\em Duke Mathematical Journal}, 51(2):459--468, 1984.

\bibitem{baake}
M.~Baake, R.~Scharlau, and P.~Zeiner.
\newblock Well-rounded sublattices of planar lattices.
\newblock {\em arXiv preprint arXiv:1311.6306}, 2013.

\bibitem{bayer}
E.~Bayer-Fluckiger.
\newblock Ideal lattices.
\newblock {\em A panorama of number theory or the view from Baker’s garden
  (Zurich, 1999)}, pages 168--184, 2002.

\bibitem{BaNe}
E.~Bayer-Fluckiger and G.~Nebe.
\newblock On the euclidean minimum of some real number fields.
\newblock {\em Journal de Th\'eorie des Nombres de Bordeaux}, 17(2):437--454,
  2005.

\bibitem{BoMa}
W.~Bola\~{n}os and G.~Mantilla-Soler.
\newblock The {T}race {F}orm {O}ver {C}yclic {N}umber {F}ields.
\newblock {\em Canad. J. Math.}, 73(4):947--969, 2021.

\bibitem{CoPe}
P.~E. Conner and R.~Perlis.
\newblock {\em A survey of trace forms of algebraic number fields}.
\newblock World Scientific, 1984.

\bibitem{conway1995lattice}
J.~H. Conway and N.~J. Sloane.
\newblock A lattice without a basis of minimal vectors.
\newblock {\em Mathematika}, 42(1):175--177, 1995.

\bibitem{conway2013sphere}
J.~H. Conway and N.~J.~A. Sloane.
\newblock {\em Sphere packings, lattices and groups}, volume 290.
\newblock Springer Science \& Business Media, 2013.

\bibitem{damir2019well}
M.~T. Damir and D.~Karpuk.
\newblock Well-rounded twists of ideal lattices from real quadratic fields.
\newblock {\em Journal of Number Theory}, 196:168--196, 2019.

\bibitem{sueli}
R.~R. de~Araujo and S.~I. Costa.
\newblock Well-rounded algebraic lattices in odd prime dimension.
\newblock {\em Archiv der Mathematik}, 112(2):139--148, 2019.

\bibitem{oliviera}
E.~L. De~Oliveira, J.~C. Interlando, T.~P. Da~Nobrega~Neto, J.~O.~D. Lopes,
  et~al.
\newblock The integral trace form of cyclic extensions of odd prime degree.
\newblock {\em Rocky Mountain Journal of Mathematics}, 47(4):1075--1088, 2017.

\bibitem{fukshansky}
L.~Fukshansky.
\newblock Well-rounded zeta-function of planar arithmetic lattices.
\newblock {\em Proceedings of the American Mathematical Society},
  142(2):369--380, 2014.

\bibitem{gnilke}
O.~W. Gnilke, H.~T.~N. Tran, A.~Karrila, and C.~Hollanti.
\newblock Well-rounded lattices for reliability and security in rayleigh fading
  siso channels.
\newblock In {\em 2016 IEEE Information Theory Workshop (ITW)}, pages 359--363.
  IEEE, 2016.

\bibitem{Damir}
A.~Karrila, M.~T. Damir, D.~Karpuk, and C.~Hollanti.
\newblock On analytical and geometric lattice design criteria for wiretap coset
  codes.
\newblock {\em arXiv preprint arXiv:1609.07723}, 2016.

\bibitem{khot2005hardness}
S.~Khot.
\newblock Hardness of approximating the shortest vector problem in lattices.
\newblock {\em Journal of the ACM (JACM)}, 52(5):789--808, 2005.

\bibitem{kuhnlein}
S.~K{\"u}hnlein.
\newblock Well-rounded sublattices.
\newblock {\em International Journal of Number Theory}, 8(05):1133--1144, 2012.

\bibitem{martinet1}
J.~Martinet.
\newblock Bases of minimal vectors in lattices, i.
\newblock {\em Archiv der Mathematik}, 89(5):404--410, 2007.

\bibitem{martinet2}
J.~Martinet.
\newblock Bases of minimal vectors in lattices, ii.
\newblock {\em Archiv der Mathematik}, 89(6):541--551, 2007.

\bibitem{martinet3}
J.~Martinet and A.~Sch{\"u}rmann.
\newblock Bases of minimal vectors in lattices, iii.
\newblock {\em International Journal of Number Theory}, 8(02):551--567, 2012.

\bibitem{mcmullen}
C.~McMullen.
\newblock Minkowski’s conjecture, well-rounded lattices and topological
  dimension.
\newblock {\em Journal of the American Mathematical Society}, 18(3):711--734,
  2005.

\bibitem{micciancio}
D.~Micciancio.
\newblock Generalized compact knapsacks, cyclic lattices, and efficient one-way
  functions from worst-case complexity assumptions.
\newblock In {\em The 43rd Annual IEEE Symposium on Foundations of Computer
  Science, 2002. Proceedings.}, pages 356--365. IEEE, 2002.

\bibitem{narkiewicz2004elementary}
W.~Narkiewicz.
\newblock {\em Elementary and Analytic Theory of Algebraic Numbers}.
\newblock Springer Monographs in Mathematics. Springer Berlin Heidelberg, 2004.

\bibitem{pohst1981computation}
M.~Pohst.
\newblock On the computation of lattice vectors of minimal length, successive
  minima and reduced bases with applications.
\newblock {\em ACM Sigsam Bulletin}, 15(1):37--44, 1981.

\bibitem{solan2019stable}
O.~N. Solan.
\newblock Stable and well-rounded lattices in diagonal orbits.
\newblock {\em Israel Journal of Mathematics}, 234(2):501--519, 2019.

\bibitem{sustik2007existence}
M.~A. Sustik, J.~A. Tropp, I.~S. Dhillon, and R.~W. Heath~Jr.
\newblock On the existence of equiangular tight frames.
\newblock {\em Linear Algebra and its applications}, 426(2-3):619--635, 2007.

\bibitem{voronoi}
G.~Voronoi.
\newblock Nouvelles applications des paramètres continus à la théorie des
  formes quadratiques. premier mémoire. sur quelques propriétés des formes
  quadratiques positives parfaites.
\newblock {\em Journal für die reine und angewandte Mathematik}, 133:97--178,
  1908.

\end{thebibliography}

\noindent

{\footnotesize Mohamed Taoufiq Damir. Department of Mathematics and Systems Analysis, Aalto University, Espoo, Finland. ({\tt mohamed.damir@aalto.fi})}

{\footnotesize Guillermo Mantilla-Soler, Department of Mathematics, Universidad Nacional de Colombia,
Medell\'in, Colombia. Department of Mathematics and Systems Analysis, Aalto University, Espoo, Finland. ({\tt gmantelia@gmail.com})}

\end{document}